\documentclass[12pt,reqno]{article}

\usepackage[usenames]{color}
\usepackage{amssymb}
\usepackage{amsmath}
\usepackage{amsthm}
\usepackage{amsfonts}
\usepackage{amscd}
\usepackage{graphicx}
\usepackage{diagbox}

\usepackage[colorlinks=true,
linkcolor=webgreen,
filecolor=webbrown,
citecolor=webgreen]{hyperref}

\definecolor{webgreen}{rgb}{0,.5,0}
\definecolor{webbrown}{rgb}{.6,0,0}

\usepackage{color}
\usepackage{fullpage}
\usepackage{float}

\usepackage{graphics}
\usepackage{latexsym}
\usepackage{epsf}
\usepackage{breakurl}

\newcommand{\seqnum}[1]{\href{https://oeis.org/#1}{\rm \underline{#1}}}
\def\Enn{\mathbb{N}}

\begin{document}

\theoremstyle{plain}
\newtheorem{theorem}{Theorem}
\newtheorem{corollary}[theorem]{Corollary}
\newtheorem{lemma}[theorem]{Lemma}
\newtheorem{proposition}[theorem]{Proposition}

\theoremstyle{definition}
\newtheorem{definition}[theorem]{Definition}
\newtheorem{example}[theorem]{Example}
\newtheorem{conjecture}[theorem]{Conjecture}

\theoremstyle{remark}
\newtheorem{remark}[theorem]{Remark}

\title{The Hurt-Sada Array and Zeckendorf Representations}

\author{Jeffrey Shallit\\
School of Computer Science \\
University of Waterloo\\
Waterloo, ON  N2L 3G1 \\
Canada \\
\href{mailto:shallit@uwaterloo.ca}{\tt shallit@uwaterloo.ca}}

\maketitle

\begin{abstract}
Wesley Ivan Hurt and Ali Sada both independently proposed
studying an infinite array where the $0$'th row
consists of the non-negative integers $0,1,2,\ldots$ in increasing order.
Thereafter the $n$'th row is formed from the $(n-1)$'th row by ``jumping''
the single entry $n$ by $n$ places to the right.  Sada also defined a
sequence $s(n)$ defined to be the first number that $n$ jumps over.
In this note I show how the Hurt-Sada
array and Sada's sequence are intimately connected
with the golden ratio $\varphi$ and Zeckendorf representation.
I also consider a number of related sequences.
\end{abstract}

\section{Introduction}

On December 9 2023, Wesley Ivan Hurt 
introduced a certain simple transformation on the 
sequence of natural numbers that generates an infinite
array having interesting properties.
Recently,
on the {\tt seqfan} mailing list,
Ali Sada \cite{Sada:2025} independently proposed studying the
same transformation, and also asked about a certain related sequence $s(n)$.
It turns out that the array and Sada's sequence are related to
$\varphi = (1+\sqrt{5})/2$, the golden ratio.  In this note
we prove some of the observed properties with the aid of
the {\tt Walnut} theorem-prover.

Let us start by arranging the natural numbers as the $0$'th row of an
infinite array $A$; see Table~\ref{tab1}.  
Row $n$ of the array $A$ is created from row $n-1$ as follows:  find the
unique occurrence of $n$ in row $n-1$, and shift this occurrence of $n$
by $n$ positions to the right.
The first few rows are given in Table~\ref{tab1}.
\begin{table}[htb]
\begin{center}
\begin{tabular}{c|ccccccccccccccccccccc}
\diagbox{$n$}{$k$} & 0& 1& 2& 3& 4& 5& 6& 7& 8& 9&10&11&12&13&14&15&16&17&18&19&\\
\hline
0 & 0& 1& 2& 3& 4& 5& 6& 7& 8& 9&10&11&12&13&14&15&16&17&18&19&\\
1 & 0& 2& 1& 3& 4& 5& 6& 7& 8& 9&10&11&12&13&14&15&16&17&18&19&\\
2 & 0& 1& 3& 2& 4& 5& 6& 7& 8& 9&10&11&12&13&14&15&16&17&18&19&\\
3 & 0& 1& 2& 4& 5& 3& 6& 7& 8& 9&10&11&12&13&14&15&16&17&18&19&\\
4 & 0& 1& 2& 5& 3& 6& 7& 4& 8& 9&10&11&12&13&14&15&16&17&18&19&\\
5 & 0& 1& 2& 3& 6& 7& 4& 8& 5& 9&10&11&12&13&14&15&16&17&18&19&\\
6 &  0& 1& 2& 3& 7& 4& 8& 5& 9&10& 6&11&12&13&14&15&16&17&18&19&\\
7 & 0& 1& 2& 3& 4& 8& 5& 9&10& 6&11& 7&12&13&14&15&16&17&18&19&\\
8 & 0& 1& 2& 3& 4& 5& 9&10& 6&11& 7&12&13& 8&14&15&16&17&18&19&\\
\end{tabular}
\end{center}
\caption{First few rows and columns of the Hurt-Sada array.}
\label{tab1}
\end{table}

It is immediate from the definition that each row is a permutation
of $\Enn$.  Therefore we can define $p(n)$ to be the position
of $n$ in row $n-1$; that is, the unique $i$ such that $A[n-1,i] = n$.
See Table~\ref{tab2}.
Sada's sequence $(s(n))_{n\geq 1}$ is defined to be the {\it first\/} of
the $n$ elements that $n$ ``jumps over''; in other words, it
is $A[n-1, p(n)+1]$.  Similarly, one can study the {\it last\/} of the
$n$ elements that $n$ ``jumps over''; in other words, it is
$t(n) := A[n-1, p(n)+n]$.  Another obvious thing to study are the
diagonal sequences $d(n) = A[n,n]$ and $d'(n) = A[n-1, n]$.
The last column in Table~\ref{tab2} gives a reference from the
On-Line Encyclopedia of Integer Sequences (OEIS) \cite{Sloane:2025}.
\begin{table}[H]
\begin{center}
\begin{tabular}{c|cccccccccccccccccccccc}
$n$ & 0& 1& 2& 3& 4& 5& 6& 7& 8& 9&10&11&12&13&14&15&16&17 & OEIS \\
\hline
$p(n)$ &0 & 1& 1& 2& 3& 3& 4& 4& 5& 6& 6& 7& 8& 8& 9& 9&10&11 & \seqnum{A060143}\\
$s(n)$ & 0 & 2& 1& 2& 5& 3& 7& 4& 5&10& 6& 7&13& 8&15& 9&10&18 & \seqnum{A380079}\\
$t(n)$ & 0 & 2& 3& 5& 7& 8&10&11&13&15&16&18&20&21&23&24&26&28 & \seqnum{A022342} \\
$d(n)$ & 0& 2& 3& 4& 3& 7& 8& 9& 6& 7&13&14& 9&10&18&19&12&13 & \seqnum{A379739} \\
$d'(n)$ & 0 & 1& 1& 2& 5& 6& 4& 5&10&11&12& 8&15&16&17&11&20&21& \seqnum{A368050}
\end{tabular}
\end{center}
\caption{Some interesting sequences.}
\label{tab2}
\end{table}

In this note we prove theorems characterizing each of these four sequences.
The proof technique uses finite automata, the Zeckendorf representations
of integers, and the {\tt Walnut} theorem prover,
and is discussed in several previous papers and the book \cite{Shallit:2023}.

\section{Zeckendorf representation}

The Fibonacci numbers are, as usual, given by
$F_0 = 0$, $F_1 = 1$, and $F_i = F_{i-1} + F_{i-2}$ for $i \geq 2$.
We briefly recall the properties of the Zeckendorf representation
(also called Fibonacci representation) of integers
\cite{Lekkerkerker:1952,Zeckendorf:1972}.  In this representation,
a non-negative integer is written as a sum of distinct Fibonacci
numbers $F_i$ for $i \geq 2$, subject to the condition that no two
adjacent Fibonacci numbers are used; the representation is unique.
The Zeckendorf representation of an integer $n$ can be written as a 
binary string $w = a_1 a_2 \cdots a_t$; here we have
$$ n = \sum_{1 \leq i \leq t} a_i F_{t+2-i}.$$
For example, the integer $43$ is represented by the
string $10010001$.  Notice that the most significant digits are
at the left, as in the case of ordinary decimal representation.

\section{Finite automata}

We will make use of finite automata, a simple model of a computer.
A finite automaton $M$ consists of a finite set of states $Q$, labeled transitions
between the states, a distinguished initial state $q_0$, and a set of
final (accepting) states $F$.  Starting in $q_0$, the automaton reads
a finite input string and follows the transitions, moving between states.
The input $w$ is said to be accepted by the automaton if, after reading 
all of $w$'s symbols, it is in a final state.  Otherwise it is rejected.
For more information about this topic, see \cite{Hopcroft&Ullman:1979}.

An automaton is depicted as a diagram, with circles representing states
and labeled arrows indicating the transitions.  The initial state is
depicted by a single arrow with no source.  Final states are indicated
by double circles.

Our automata will read the Zeckendorf representation of integers as inputs.
Hence we can regard the automata as taking integer inputs, instead of
strings.  Sometimes we will need automata to take multiple integer inputs
at once.  In this case, the Zeckendorf representations of the integers
are read {\it in parallel}, with the shorter representations (if there are
any) padded with zeros at the front to make them agree in length.

\section{The {\tt Walnut} theorem prover}

We will also make use of the {\tt Walnut} theorem prover, originally
designed by
Hamoon Mousavi \cite{Mousavi:2016}, and described in more detail
in \cite{Shallit:2023}.  This free software package can rigorously prove or disprove
first-order claims about automata and sequences of integers they accept.

A brief discussion of {\tt Walnut}'s syntax now follows.  The main
commands are {\tt def} and {\tt eval}; the first defines an automaton
for future use, and the second evaluates a logical assertion with
no free variables as either {\tt TRUE} or {\tt FALSE}.  Once an automaton
is defined, it can be used by prefixing its name with a dollar sign.
The basic logical operations are AND (represented by {\tt \&}),
OR (represented by {\tt |}), NOT (represented by {\tt \char'176}),
implication (represented by {\tt =>}), and IFF (represented by
{\tt <=>}).   The basic arithmetic operations are addition, subtraction,
and multiplication by a constant.  Integer division by a constant
is also allowed.  The universal quantifier $\forall$ is written
{\tt A}, and the existential quantifier $\exists$ is written {\tt E}.
The statement {\tt ?msd\_fib} at the beginning of a {\tt Walnut} command
specifies that inputs should be understood to represent integers
represented in Zeckendorf representation.

As a simple example, the following {\tt Walnut} command defines
an automaton {\tt even} that checks if its input in Zeckendorf
representation represents an even integer.
\begin{verbatim}
def even "?msd_fib Ek n=2*k":
\end{verbatim}
It produces the automaton depicted in Figure~\ref{fig0}.
\begin{figure}[htb]
\begin{center}
\includegraphics[width=5in]{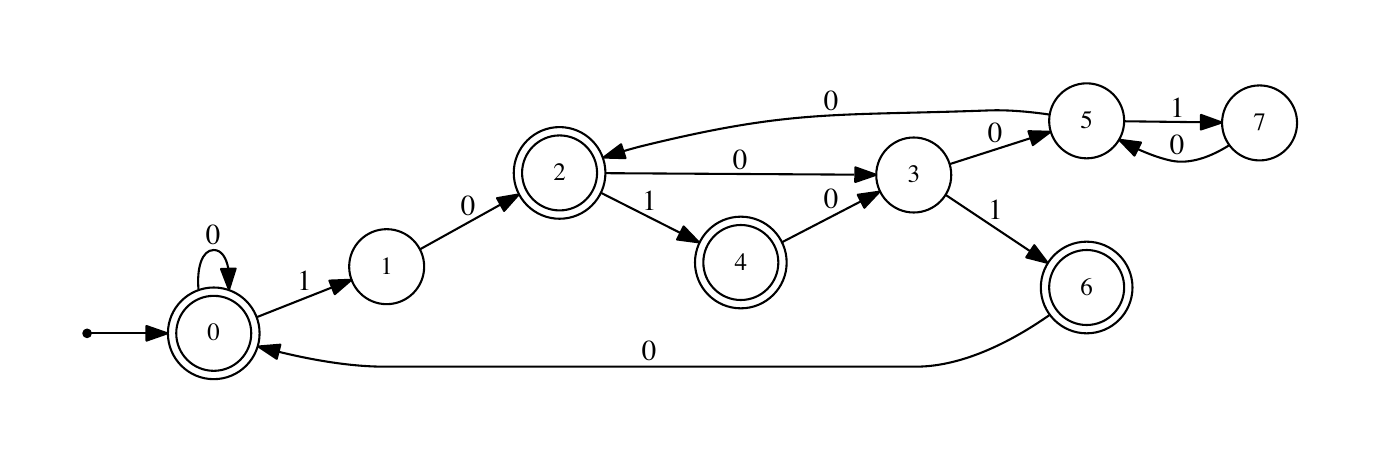}
\caption{Automaton accepting even numbers in Zeckendorf representation.}
\label{fig0}
\end{center}
\end{figure}

We will make use of several {\tt Walnut} automata which were derived
in \cite{Shallit:2023}:
\begin{itemize}
\item {\tt phin}$(n,x)$ accepts iff $x = \lfloor \varphi n \rfloor$;
\item {\tt phi2n}$(n,x)$ accepts iff $x = \lfloor \varphi^2 n \rfloor$;
\item {\tt noverphi}$(n,x)$ accepts iff $x = \lfloor n/\varphi \rfloor$.
\end{itemize}

\section{The theorems}

The fundamental theorem on which everything is based is the following:
\begin{theorem}
There is a $52$-state automaton $\tt m$ that takes three inputs $x,y,z$
expressed in Zeckendorf representation, and accepts
if and only if $A[x,y] = z$.
\end{theorem}

\begin{proof}
We ``guessed'' the automaton {\tt m} using a variant of the Myhill-Nerode theorem,
based on empirical data.  (We used the first 291 rows and 471 columns of
the Hurt-Sada array.)   
This automaton is too large to display here, but it can
be obtained from the author's web page.

Briefly, this is how the guessing procedure 
works: you start with a language $L$ for which you can determine membership, and you suspect $L$ is regular (accepted by an automaton).
In the Myhill-Nerode theorem, states are associated with equivalence classes of input strings; two strings $x$ and $y$ are equivalent if for all strings
$z$ we have $xz \in L$ iff $yz \in L$.   
Then a minimal automaton for the language $L$
is built out of the equivalence classes.  
To heuristically guess an automaton, just change the definition of 
equivalence: now $x$ and $y$ are deemed to be equivalent if $xz \in L$ iff $yz \in L$ for all strings $z$ of length $\leq k$.  Then you 
build the guessed automaton out of the equivalence classes, as usual.
You now do this for larger and larger $k$.
If the number of states stabilizes, you have a candidate automaton.
For more details, see \cite[\S 5.6]{Shallit:2023}.

It now remains to verify that our guessed automaton is correct.
The first step is to establish that {\tt m} the automaton really does
compute a function of the inputs $x$ and $y$.  In other words, for every
$x$ and $y$ there must be exactly one $z$ such that the triple
$(x,y,z)$ is accepted.

We can verify this with the following {\tt Walnut} code.
\begin{verbatim}
eval check_fn1 "?msd_fib Ax,y Ez $m(x,y,z)": 
eval check_fn2 "?msd_fib Ax,y,z1,z2 ($m(x,y,z1) & $m(x,y,z2)) => z1=z2":
\end{verbatim}
and {\tt Walnut} returns {\tt TRUE} for both.

The next step is to verify that indeed {\tt m} computes the Hurt-Sada array.
To do so we use induction on the row number $n$.  The base case, the correctness
of row $0$, can be verified using the following code:
\begin{verbatim}
eval sada_c1 "?msd_fib An $m(0,n,n)":
\end{verbatim}

Now we verify that for each $n$, row $n$ follows from row $n-1$ using
the Hurt-Sada transformation.  Namely, we must check that
\begin{itemize}
\item[(a)] $A[n,j]=A[n-1,j]$ for $0\leq j < i$.
\item[(b)] $A[n,p(n)+n] = n$.
\item[(c)] $A[n,j] = A[n-1,j+1]$ for $p(n) \leq j < p(n)+n$.
\item[(d)] $A[n,j] = A[n-1,j]$ for $j > n+ p(n)$.
\end{itemize}
We can check each condition using the following code:
\begin{verbatim}
def p "?msd_fib $m(n-1,z,n)":
eval case_a "?msd_fib An,i,j,x ($p(n,i) & j<i & $m(n-1,j,x)) => $m(n,j,x)":
eval case_b "?msd_fib An,x,y ($p(n,x) & $m(n,x+n,y)) => y=n":
eval case_c "?msd_fib An,x,y,j ($m(n-1,j+1,x) & $p(n,y) & y<=j & j<y+n)
   => $m(n,j,x)":
eval case_d "?msd_fib An,j,x,z ($m(n-1,j,x) & $p(n,z) & j>n+z) => $m(n,j,x)":
\end{verbatim}
and all return {\tt TRUE}.
\end{proof}

At this point we know the correctness of the automaton {\tt m} computing
the Hurt-Sada array $A$.
We can now use this automaton to rigorously derive and
prove many conjectures about the elements, provided they can be stated
in first-order logic.

As a warmup, let us prove 
\begin{theorem}
Every integer $n$ eventually returns to its starting position in column $n$;
that is, for each $n$ there exists $m$ such that $A[x,n] = n$ for all $x\geq m$.
\end{theorem}

\begin{proof}
We use the following {\tt Walnut} code.
\begin{verbatim}
eval return "?msd_fib An Em Ax (x>=m) => $m(x,n,n)":
\end{verbatim}
And {\tt Walnut} returns {\tt TRUE}.
\end{proof}

Now that we know that each integer eventually returns to its starting position,
we can ask about where this return happens for the first time.

\begin{theorem}
The integer $n$ first leaves the $n$'th column at row $\lfloor (n+1)/\varphi \rfloor$,
and returns to the $n$'th column at row $\lfloor (n+1)\varphi \rfloor - 1$.
From then on $n$ remains in the $n$'th column.
\end{theorem}

\begin{proof}
We use the following code.
\begin{verbatim}
def leaves "?msd_fib (Ai i<z => $m(i,n,n)) & ~$m(z,n,n)":
eval leaves_thm "?msd_fib An,z $leaves(n,z) => $noverphi(n+1,z)":
def returns "?msd_fib (Aj (j>=z) => $m(j,n,n)) & ~$m(z-1,n,n)": 
eval returns_thm "?msd_fib An,z $returns(n,z) => $phin(n+1,z+1)":
\end{verbatim}
And {\tt Walnut} returns {\tt TRUE} for all of them.
\end{proof}

Let us now find a simple formula for $p(n)$, the location of $n$ in row $n-1$:
\begin{theorem}
We have $p(n) = \lfloor (n+1)/\varphi \rfloor$.
\end{theorem}
\begin{proof}
We use the following code:
\begin{verbatim}
def p "?msd_fib $m(n-1,z,n)":
eval p_check "?msd_fib An,z $p(n,z) => $noverphi(n+1,z)":
\end{verbatim}
\end{proof}

We can now find a simple closed form for Sada's sequence $2,1,2,5,\ldots$.
Recall that $\{x \} := x \bmod 1$, the fractional part of $x$.
\begin{theorem}
If $\{ (n+1)\varphi \} < 2-\varphi$, then $s(n) = n+1$.  Otherwise,
$s(n) = \lfloor (n+1)/\varphi \rfloor$.
\end{theorem}

\begin{proof}
Note that $0 \leq (n+1)\varphi < 2-\varphi$ if and only if
there exists $k$ such that $n = \lfloor k \varphi \rfloor - 1$.
We can now verify the claim with the following {\tt Walnut} code.
The first line defines an automaton, {\tt sad}, which computes
$s(n)$.
\begin{verbatim}
def sad "?msd_fib Ex $p(n,x) & $m(n-1,x+1,z)":
eval thma "?msd_fib An (Ek $phi2n(k,n+1)) <=> $sad(n,n+1)":
eval thmb "?msd_fib An (n>0) => 
   ((~Ek $phi2n(k,n+1)) <=> (Ex $sad(n,x) & $noverphi(n+1,x)))":
\end{verbatim}
\end{proof}

\begin{theorem}
No integer appears three times or more in $(s(n))_{n \geq 1}$.
Furthermore, $n$ appears exactly twice if and only if 
$n$ belongs to sequence \seqnum{A001950}, that is, if and only if there
exists $k \geq 1$ such that $n = \lfloor k \varphi^2 \rfloor$.
\end{theorem}

\begin{proof}
To check the first assertion we use the following code:
\begin{verbatim}
eval no3 "?msd_fib ~Ei,j,k,n i<j & j<k & $sad(i,n) & $sad(j,n) & $sad(k,n)":
\end{verbatim}
which returns {\tt TRUE}.

For the second we write
\begin{verbatim}
eval twice "?msd_fib An (Ei,j i<j & $sad(i,n) & $sad(j,n)) <=> 
   (Ek k>=1 & $phi2n(k,n))":
\end{verbatim}
\end{proof}

We now turn to studying the sequence $t(n)$.
\begin{theorem}
We have $t(n) = \lfloor (n+1) \varphi \rfloor - 1$.
\end{theorem}

\begin{proof}
We use the following code:
\begin{verbatim}
def t "?msd_fib Ex $p(n,x) & $m(n-1,n+x,z)":
eval test_t "?msd_fib An,x,y ($t(n,x) & $phin(n+1,y)) => x+1=y":
\end{verbatim}
and {\tt Walnut} returns {\tt TRUE}.
\end{proof}

We now study the diagonal sequence $d(n)$.  Values of this sequence
split naturally into two kinds:  those $n$ for which $d(n) \geq n$
and those $n$ for which $d(n) < n$.  The former have
$d(n)$ approximately equal to $1.236n$, while the latter have
$d(n)$ approximately equal to $.7639n$.  This is made more precise
in the next theorem.

\begin{theorem}
There is an $8$-state automaton that decides whether
$d(n) \geq n$ or $d(n) < n$.  It is depicted in Figure~\ref{dlgn}.

Further the diagonal sequence $d(n) = A[n,n]$ satisfies
$d(n) < n  \implies d(n) = \lfloor (4-2\varphi)n + (5-3\varphi) \rfloor$ and
$d(n) \geq n \implies d(n) = \lfloor (2\varphi-2)n\rfloor + 1$.
\end{theorem}

\begin{figure}[htb]
\begin{center}
\includegraphics[width=5in]{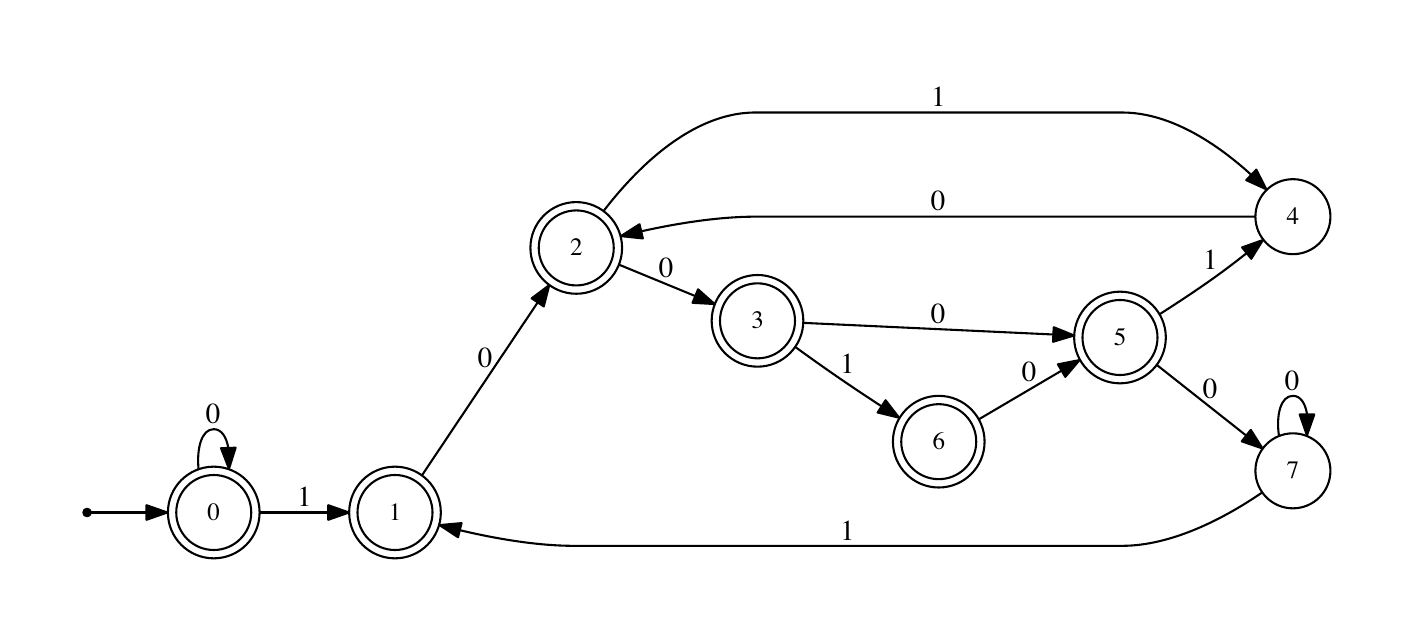}
\end{center}
\caption{The automaton that decides if $d(n)\geq n$.}
\label{dlgn}
\end{figure}

\begin{proof}
We use the following {\tt Walnut} code.
\begin{verbatim}
def d "?msd_fib $m(n,n,z)":
def dlgn "?msd_fib Ex $d(n,x) & x>=n":
\end{verbatim}
This creates the automaton in Figure~\ref{dlgn}.

Now we need to verify two assertions.
For the first assertion, we use the following code:
\begin{verbatim}
eval dltn "?msd_fib An,x,y (~$dlgn(n) & $d(n,x) & $phin(2*n+3,y)) 
   => x+y=4*n+4":
\end{verbatim}
The second can be verified as follows:
\begin{verbatim}
eval dgn "?msd_fib An,x,y (n>=1 & $dlgn(n) & $d(n,x) & $phin(2*n,y)) 
   => x+2*n=y+1":
\end{verbatim}
\end{proof}

In exactly the same way we can prove a theorem about the diagonal
sequence $d'(n)$.  
\begin{theorem}
There is a 6-state automaton that decides whether
$d'(n) \geq n$.  It is depicted in Figure~\ref{dpfig}.  If $d'(n)\geq n$ then
$d'(n) = \lfloor (2\varphi - 2)n + {1\over 2} \rfloor$.
If $d'(n) < n$ then
$d'(n) = \lfloor (4-2 \varphi)n \rfloor$.
\end{theorem}

\begin{figure}[htb]
\begin{center}
\includegraphics[width=5in]{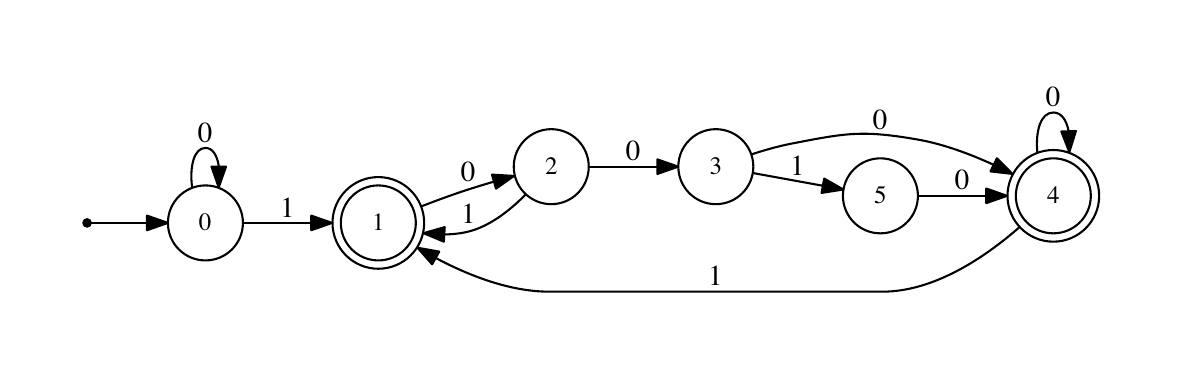}
\end{center}
\caption{The automaton that decides if $d'(n)\geq n$.}
\label{dpfig}
\end{figure}

\begin{proof}
We use the same idea as before.
\begin{verbatim}
def dp "?msd_fib $m(n-1,n,z)":
def dpg "?msd_fib Ez $dp(n,z) & z>=n":
eval dpgn "?msd_fib An,x,y (n>=1 & $dpg(n) & $dp(n,x) & $phin(4*n,y))
   => x=((y+1)-4*n)/2":
eval dpltn "?msd_fib An,x,y (n>=1 & ~$dpg(n) & $dp(n,x) & $phin(2*n,y))
   => x+1+y=4*n":
\end{verbatim}
\end{proof}

\section{Row characteristics}

Clearly, for all $n$,
the $n$'th row of $A$ consists of three different regions:  an 
initial segment, where $i$ appears in column $i$ for $0 \leq i \leq b(n)$;
a middle segment that consists of some permutation
of $b(n)+1, \ldots, c(n)-1$,
and a final (infinite) segment, where again $i$ appears in column $i$
for $c(n) \leq i < \infty$.  Furthermore, it appears that the permutation
in the middle segment leaves no element fixed.  We now find
a simple formula for $b(n)$, $c(n)$, and prove these claims.

\begin{theorem}
We have $b(n) = \lfloor (n+2)/\varphi \rfloor - 1$ for $n \geq 1$.
We also have $c(n) = \lfloor (n+1)\varphi \rfloor$ for $n \geq 1$.
Finally, in row $n$, the elements in columns
$b(n)+1,\ldots, c(n)-1$ form a permutation of
those elements, a permutation that leaves no element fixed.
\end{theorem}

\begin{proof}
We start with automata that compute $b(n)$ and $c(n)$:
\begin{verbatim}
def b "?msd_fib (~$m(n,z+1,z+1)) & At (t<=z) => $m(n,t,t)":
def c "?msd_fib (~$m(n,z-1,z-1)) & At (t>=z) => $m(n,t,t)":
\end{verbatim}
Now we prove the first two claims of the theorem.
\begin{verbatim}
eval test_b "?msd_fib An,x,y (n>=1 & $noverphi(n+2,x) & $b(n,y)) => y+1=x":
eval test_c "?msd_fib An,x,y (n>=1 & $phin(n+1,x) & $c(n,y)) => y=x":
\end{verbatim}
Next we prove the last two claims.
\begin{verbatim}
eval test_perm1 "?msd_fib An,x,y,i (n>=1 & $b(n,x) & $c(n,y) & i>x & i<y) =>
   Ej j>x & j<y & $m(n,i,j)":
eval no_fixed_point "?msd_fib ~En,x,y,i n>=1 & $b(n,x) & $c(n,y) & i>x
   & i<y & $m(n,i,i)":
\end{verbatim}
\end{proof}

\section{Antidiagonals}

With the automaton {\tt m} one can very easily explore other facets of 
the Hurt-Sada array.   For example, consider the $n$'th antidiagonal, that is,
the elements $A[0,n], A[1,n-1], A[2,n-2], \ldots, A[n,0]$.

\begin{theorem}
There are functions $h,h',r$ with the following properties.
Every antidiagonal of the Hurt-Sada array starts with
$n$ in row $0$ and decreases by $1$ with each succeeding row up to
row $h(n)-1$.
Then it takes the value $r(n)$ for consecutive
rows $h(n), h(n)+1, \ldots, h'(n)$.
Then, starting at row $h'(n)+1$, 
it decreases again by $1$ at each step until it hits $0$ in row $n$.
\end{theorem}

\begin{proof}
We start by showing that each antidiagonal is a decreasing sequence.
\begin{verbatim}
eval decreasing "?msd_fib Ai,n,x,y (i<n & $m(i,n-i,x) &
   $m(i+1,n-(i+1),y)) => x>=y":
\end{verbatim}
Next we create an automaton that accepts $(n,x)$ if $x$ appears at least
twice in the $n$'th antidiagonal.
\begin{verbatim}
def occurs_twice "?msd_fib Ei,j i<j & j<=n & $m(i,n-i,x) & $m(j,n-j,x)":
\end{verbatim}
Next we check that an antidiagonal has at most one such element,
and every antidiagonal of index $2$ or more has at least one.
This shows it is unique for $n\geq 2$.
\begin{verbatim}
eval never2 "?msd_fib An ~Ex,y x!=y & $occurs_twice(n,x) & $occurs_twice(n,y)":
eval has1 "?msd_fib An (n>=2) => Ex $occurs_twice(n,x)":
\end{verbatim} 
We can now define the function $r$, where we arbitrarily choose
$r(0) = 0$ and $r(1) = 1$.
\begin{verbatim}
def r "?msd_fib (n=0&z=0)|(n=1&z=1)|(n>1 & $occurs_twice(n,z))":
\end{verbatim}
Similarly, we can now define the functions $h$ and $h'$:
\begin{verbatim}
def h "?msd_fib (n<=1&i=0) | (n>=2 & Ex $occurs_twice(n,x) & $m(i,n-i,x) & 
   Aj (j<i) => ~$m(j,n-j,x))":
def hp "?msd_fib (n<=1&i=0) | (n>=2 & Ex $occurs_twice(n,x) & $m(i,n-i,x) &
   Aj (j>i & j<=n) => ~$m(j,n-j,x))":
\end{verbatim}
Now we check the decreasing property:
\begin{verbatim}
eval check_decreasing1 "?msd_fib An,i,x,y,z (n>=2 & $h(x,n) &
   i<=x-2 & $m(i,n-i,y) & $m(i+1,n-(i+1),z)) => y=z+1":
eval check_decreasing2 "?msd_fib An,i,x,y,z (n>=2 & $hp(x,n) &
   i>x & i<n & $m(i,n-i,y) & $m(i+1,n-(i+1),z)) => y=z+1":
\end{verbatim}

\begin{table}[htb]
\begin{center}
\begin{tabular}{c|ccccccccccccccccccccccc}
$n$ & 0& 1& 2& 3& 4& 5& 6& 7& 8& 9&10&11&12&13&14&15&16&17&18&19& 20\\
\hline
$r(n)$ &  0& 1& 2& 1& 3& 2& 4& 5& 3& 6& 7& 4& 8& 5& 9&10& 6&11& 7&12&13&\\
$h(n)$ &  0&0&0&1&1&2&2&2&3&3&3&4&4&5&5&5&6&6&7&7&7&\\
$h'(n)$ & 0& 0& 1& 2& 2& 3& 3& 4& 5& 5& 6& 7& 7& 8& 8& 9&10&10&11&11&12&\\ 
\end{tabular}
\end{center}
\caption{The functions $r,h,h'$.}
\label{tab4}
\end{table}
\end{proof}

When we look up $r(n)$ in the OEIS, we find a match with
\seqnum{A026272}.   
Similarly, when we look up $h'(n)$ we find
a match for \seqnum{A319433}.
Let us now prove these are really identical!

The definition of \seqnum{A026272} is as follows:
$a(n)$ is the smallest integer $k$ such that $k = a(n-k-1)$ is the only
appearance of $k$ so far; if there is no such $k$ then
$a(n)$ = least positive integer that has not appeared yet.

We can verify that our sequence $r(n)$ has this property:
\begin{verbatim}
def prop "?msd_fib $r(n-(k+1),k) & ~Ei,j i<j & j<n & $r(i,k) & $r(j,k)":
def part1 "?msd_fib $prop(k,n) & Aj (j<k) => ~$prop(j,n)":
eval test1 "?msd_fib Ak,n (n>=1 & k>=1 & $part1(k,n)) => $r(n,k)":
def occurs "?msd_fib Ei i<n & $r(i,z)":
def mex "?msd_fib (~$occurs(n,z)) & At (t<z) => $occurs(n,t)":
eval test2 "?msd_fib An (~Ek $part1(k,n)) => Es $mex(n,s) & $r(n,s)":
\end{verbatim}
An automaton for $r(n)$ is given in Figure~\ref{figr}. It takes
pairs $(n,x)$ as input and accepts iff $x = r(n)$.
\begin{figure}[htb]
\begin{center}
\includegraphics[width=5in]{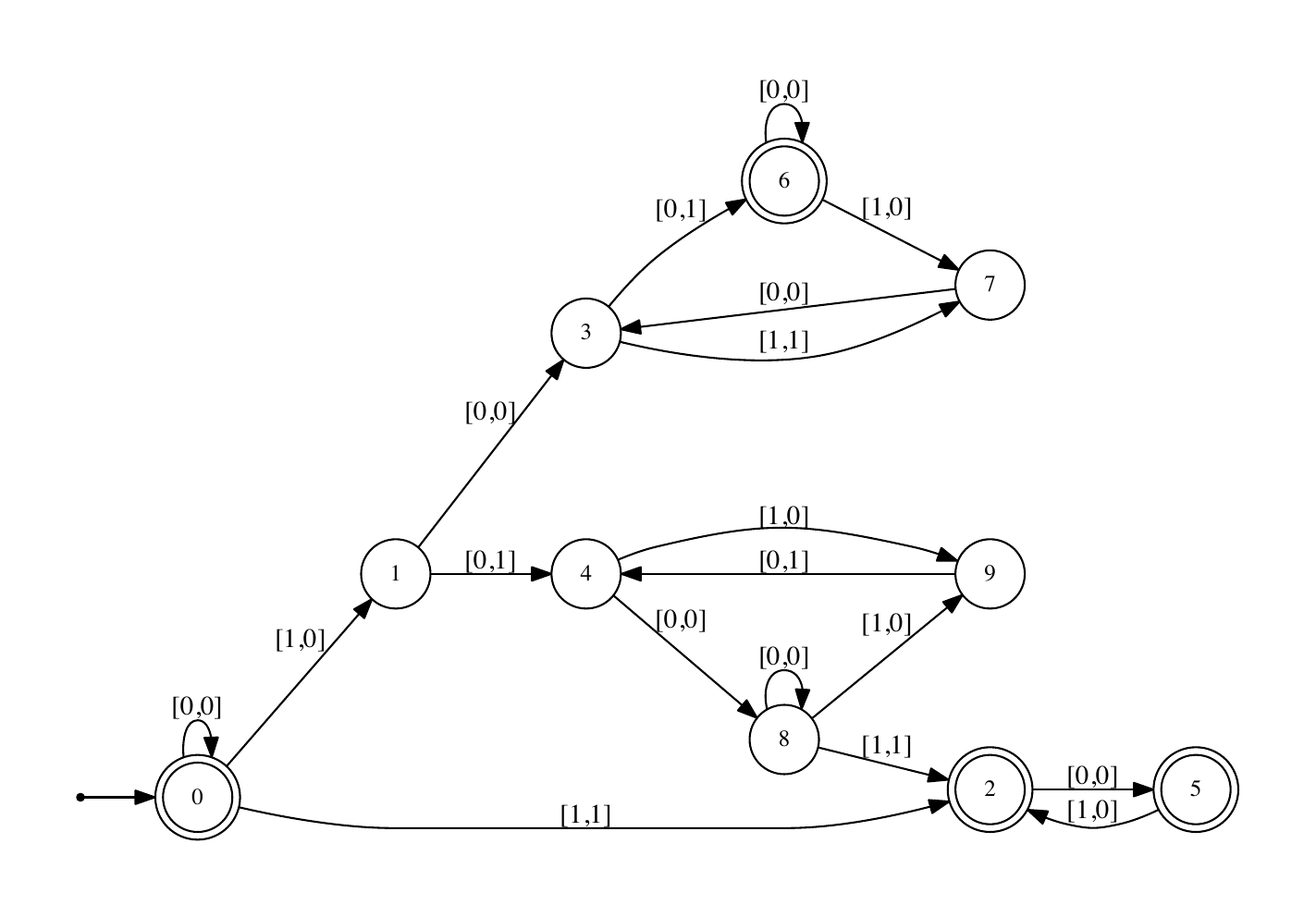}
\end{center}
\caption{Automaton that computes $r(n)$.}
\label{figr}
\end{figure}

We now establish a formula for $h(n)$, namely,
$h(n) =  2n+4-\lfloor (n+3)\varphi \rfloor$.
\begin{verbatim}
eval h_formula "?msd_fib An,z $h(z,n) <=> Ex $phin(n+3,x) & z+x=2*n+4":
\end{verbatim}

Finally, we get a formula for $h'(n)$, namely,
the one given in \seqnum{A319433}:  $h'(n) = \lfloor (n+2)/\varphi \rfloor - 1$.
We can verify this as follows:
\begin{verbatim}
eval hp_formula "?msd_fib An,z $hp(z,n) <=> Ex $noverphi(n+2,x) & z+1=x":
\end{verbatim}

\section{A final word}

No doubt the Hurt-Sada array contains many other interesting sequences that one
can study.  With the aid of the automaton {\tt m} and the {\tt Walnut} package,
one can easily explore its properties and discover new theorems.

For more about {\tt Walnut}, see the website\\
\centerline{\url{https://cs.uwaterloo.ca/~shallit/walnut.html} .}

\end{document}